\documentclass[11pt]{amsart}

\usepackage{tikz}

\usepackage{amssymb}
\usepackage{amsmath}
\usepackage{amsfonts}
\usepackage{graphicx}

\usepackage[utf8]{inputenc}
\usepackage[T1]{fontenc}
\usepackage{textcomp}
\usepackage{indentfirst}

\usepackage[pagewise]{lineno}

\usepackage[total={17cm,22cm},top=2.5cm, left=2.3cm]{geometry}
\def\XXint#1#2#3{{\setbox0=\hbox{$#1{#2#3}{\int}$}
\vcenter{\hbox{$#2#3$}}\kern-.5\wd0}}

\usepackage{latexsym,amsmath,amssymb,amsfonts,cancel}

\usepackage{hyperref}
    \usepackage{aeguill}
    \usepackage{type1cm}

\usepackage[shortlabels]{enumitem}

\theoremstyle{plain}

\newtheorem{thm}{Theorem}[section]
\newtheorem{cor}[thm]{Corollary}
\newtheorem{lem}[thm]{Lemma}
\newtheorem{fact}[thm]{Fact}

\newtheorem{rem}[thm]{Remark}

\newtheorem{defn}[thm]{Definition}

\newtheorem{claim}[thm]{Claim}

\usepackage{amsxtra}
\usepackage{eucal}
\usepackage{mathrsfs}
\usepackage{longtable}
\usepackage{caption}

\newcommand{\N}{\mathbb{N}}
\newcommand{\Z}{\mathbb{Z}}

\newcommand{\R}{\mathbb{R}}

\newcommand{\dist}{\operatorname{dist}}

\newcommand{\espan}{\operatorname{span}}
\newcommand{\Ker}{\operatorname{Ker}}

\def\entier{\,\,\hbox{{\rm N}\kern-.9em\hbox{{\rm I}}}\,\,\,}
\def\reel{\,\,\hbox{{\rm R}\kern-.9em\hbox{{\rm I}}}\,\,\,}
\def\sn{\smallskip\noindent}
\def\mn{\medskip\noindent}

\begin{document}

\title{Normal tilings of a Banach space and its ball}

\author{Robert Deville}
\address{Institut de Mathématiques de Bordeaux, Université de Bordeaux 1, 33405, Talence, France
} \email{robert.deville@math.u-bordeaux.fr
}

\author{Miguel García-Bravo}
\address{ICMAT (CSIC-UAM-UC3-UCM), Calle Nicol\'as Cabrera 13-15.
28049 Madrid, Spain} \email{miguel.garcia@icmat.es}

\thanks{M. Garc\'ia-Bravo was supported by Programa Internacional de Doctorado Fundación La Caixa–Severo Ochoa 2016 and partially by MTM2015-65825-P}

\date{January, 2020}

\keywords{Normal tilings, starshaped, uniform convex spaces}

\begin{abstract}
We show some new results about tilings in Banach spaces. A tiling of a Banach space $X$ is a covering by closed sets with non-empty interior so that they have pairwise disjoint interiors. If moreover the tiles have inner radii uniformly bounded from below, and outer radii uniformly bounded from above, we say that the tiling is normal.

In 2010 Preiss constructed a convex normal tiling of the separable Hilbert space. For Banach spaces with Schauder basis we will show that Preiss' result is still true with starshaped tiles instead of convex ones. Also, whenever $X$ is uniformly convex we give precise constructions of convex normal tilings of the unit sphere, the unit ball or in general of any convex body.
\end{abstract}

\maketitle

\section{Introduction} 

\begin{defn}
{\em Let $X$ be a Banach space.  A {\em tiling} of $X$ is a family  $\{C_{\alpha}\}$ of closed sets with non-empty interior that covers  $X$, and such that, if $\alpha\ne \beta$,
then $C_\alpha\cap C_\beta$ has empty interior. The tiling is said to be convex if the tiles are convex and starshaped if the tiles are starshaped. And the tiling is normal if there exists some $R>r>0$ so that each tile $C_{\alpha}$ contains a point $c_\alpha$ for which  $B(c_\alpha,r)\subset C_\alpha\subset B(c_\alpha,R)$.} 
\end{defn}

To get a normal tiling by convex sets of the finite dimensional Hilbert space, one just have to take a maximal family $D=\{d_i:\,i\in \N\}$ satisfying $\Vert d_i-d_j\Vert\geq 2$ for $i\neq j$ and define the Voronoi cells
$$
C_i=\{x\in \R^n:\,\Vert x-d_i\Vert=\dist (x,D)\}.$$
Then one can check that $B(d_i,1)\subset C_i\subset B(d_i,2)$ for every $i\in \N$. We must note two facts: this construction is not possible in infinite dimensions since we lack compactness of the unit ball; and for other non-hilbertian  norms in $\R^n$ the Voronoi cells are not necessarily convex sets.

\medskip

The theory of tilings in infinite-dimensional normed spaces is much less developed than the one in finite dimensions. The first one who considered this problem in infinite dimensions for the case of non-separable Banach spaces was Klee in 1981 (\cite{Klee1} and see also \cite[pp. 423]{Klee2}). He proved that the spaces $\ell_p(\Gamma)$ where $\Gamma$ is an infinite cardinal  such that $\Gamma^{\aleph_0}=\Gamma$, 
contain a $2^{1/p}-$dispersed proximinal set $D$ (that is, every $x\in \ell_p(\Gamma)$ has a, not necessarily unique, nearest point in $D$). Then the Voronoi cells defined as 
$$C_d=\{x\in \ell_p(\Gamma):\,\dist (x,d)=\dist (x,D)\}, $$
with $d\in D$, define a normal tiling made out of unit balls. In the case of $\ell_1(\Gamma)$ the set of centres form a Chebyshev set meaning that all the balls of our tiling are pairwise disjoint. This was the first example of a discrete Chebyshev set in a normed space of infinite dimension.

For the case of separable Banach spaces things are more complicated. An easy case is that of $c_0$: just consider the set $A=\{(a_n)\in c_{00}:\,a_n\in 2\Z\}$ and the balls $C_a=\{x\in c_0:\,\Vert x-a\Vert_\infty \leq 1\}$. Then $\{C_a\}_{a\in A}$ form a convex normal tiling which is moreover locally finite. However, by a result due to Corson \cite{Corson}, there do not exist locally finite tilings by  bounded convex sets in reflexive spaces. This fact was generalized in \cite{Fonf}, where it was proved that for separable Banach spaces, admitting locally finite tilings by bounded convex sets is equivalent to being isomorphic to a polyhedral space. 
Furthermore, for a Hilbert space, where the construction through Voronoi cells starting from a proximinal discrete set worked well in finite dimensions and in the non-separable case (Klee's result), is impossible in the separable case (see \cite[Proposition 2.1]{FoLi}).

These previous observations tell us that finding normal tilings in separable Banach spaces of infinite dimensions should require an involved construction.

\medskip

In the search of normal tilings in infinite-dimensional Banach  spaces we should highlight the result of Fonf, Pezzotta and Zanco, \cite[Theorem 2.2]{FPZ}. They showed that any normed space admits a tiling by bounded convex sets, with uniformly bounded inner radii. As an easy consequence, for the case of spaces with the Radon Nikodym property, in \cite[Proposition 2.8]{FoLi} it is pointed out that we can get tilings by convex sets with uniformly bounded outer radii (but not uniformly inner radii at the same time).

It was in 2010 when D. Preiss, in his paper \cite{Preiss}, found the first example of a convex normal tiling  of the separable Hilbert space. However the question whether we can find such tilings for other separable Banach spaces like $\ell_p$, $p\neq 2$, remains open.

\medskip

In this paper, following Preiss' construction, we prove that for any Banach space having Schauder basis there exists a normal tiling by starshaped tiles. This is done in Section \ref{Tiling Banach spaces with Schauder basis}.

We are also interested in finding tilings of other type of sets, different of the whole space. In particular, for the case of uniformly convex Banach spaces $X$ we can find a normal tiling of the unit sphere $\{S_\alpha\}\subset S_X$ whose tiles are the intersection of convex sets with the unit sphere. We will also show that it is possible to find convex normal tilings of the unit ball, or in general of any bounded convex set with non-empty interior. We build it up by creating uniform slices of the convex body, where we control both the inner and outer radius of the slice at the same time. This corresponds to Sections \ref{Tiling of the unit sphere} and \ref{Tiling of convex bodies} respectively. 

Finally in Section \ref{Normal tiling of the unit ball of $L_1$} we give a normal tiling of the unit sphere and unit ball of $L_1$, with tiles not necessarily convex or starshaped.

\section{Tiling Banach spaces with Schauder basis}\label{Tiling Banach spaces with Schauder basis}

\begin{thm}\label{tiling Schauder basis} Let $X$ be a Banach space having a Schauder basis $(e_k)_{k\geq 0}$.
Then $X$ admits a starshaped normal tiling $\{C_n\}_{n\geq 1}$.
\end{thm}
This theorem has been proved by David Preiss whenever $X$ is a separable Hilbert space. 
In that case, the tiles were convex. The tiles obtained whenever $X$ is a
separable Banach space having a Schauder basis are no longer necessarily convex, but they remain starshaped.
The proof presented below is a variant of the proof given by David Preiss.
 
We first present two lemmas:

\begin{lem}\label{tiling of V_k}
If $V$ is a finite dimensional normed space and $r>0$,  
there exists a starshaped tiling $\{D_i:\,i\in\N\}$ of $V$ and points 
$d_i\in D_i$ such that for each $i$, 
$B(d_i,r)\subset D_i\subset B(d_i,2r)$. 
\end{lem}
\begin{proof} Let $\{d_i:\,i\in\N\}$ be a maximal family in $V$ such that whenever $i\ne j$,
$\Vert d_i-d_j\Vert\ge 2r$. The Voronoi sets are the sets
$V_i=\bigcap_j\{x\in V:\,\Vert x-d_i\Vert\le \Vert x-d_j\Vert\}$, which are closed, form a covering of $V$, and for each $i\in\N$, $B(d_i,r)\subset V_i\subset B(d_i,2r)$. 

If $V$ is a finite dimensional Hilbert space, the sets $V_i$ form the desired tiling. Indeed, in this case, the tiles $V_i$ have pairwise disjoint interior and are convex: $V_i$ is the intersection of the half spaces  $\{x\in V:\,(d_j-d_i,x)\le(\Vert d_j\Vert^2-\Vert d_i\Vert^2)/2\}$. 

If $V$ is an arbitrary finite dimensional normed space, we define 
$$D_i=\overline{V_i\setminus \bigcup_{j<i}V_j}.$$
The sets $D_i$ are closed, they form a covering of $V$ and for each $i\in\N$, $B(d_i,1)\subset D_i\subset B(d_i,2)$. Also the tiles $D_i$ have pairwise disjoint interiors. In fact if $j<i$,

$$\buildrel\circ\over D_i \cap \buildrel\circ\over D_j\subset \left( \overline{V_i\setminus V_j}\right) \cap \buildrel\circ\over V_j=\emptyset.$$
To check the last equality take $x\in \buildrel\circ\over V_j$, let $r>0$ so that $\buildrel\circ\over B(x,r)\subset V_j$, and hence $\overline{V_i\setminus V_j}\subset V_i\setminus \buildrel\circ\over B(x,r)$, which means that $x\notin \overline{V_i\setminus V_j}$.

To end, let us prove that $D_i$ is starshaped with respect to $d_i$, for which is enough that $V_i\setminus \bigcup_{j<i}V_j$ is starshaped. Take $x\in V_i\setminus \bigcup_{j<i}V_j$ and $t\in [0,1]$.
\begin{itemize}
\item $td_i+(1-t)x\in V_i$ because $V_i$ is starshaped with respect to $d_i$ and $x\in V_i$.
Let us prove that  $V_i$ is starshaped with respect to $d_i$. Without loss of generality, we can assume that $d_i=0$. 
If $x\in V_i$ and if $s\in [0,1]$, then, for each $j\in\N$, $\Vert x\Vert\le\Vert x-d_j\Vert$, and so:
$$
\Vert sx\Vert=(s-1)\Vert x\Vert+\Vert x\Vert\le
(s-1)\Vert x\Vert+\Vert (1-s)x+sx-d_j\Vert\le
\Vert sx-d_j\Vert,
$$
therefore $sx\in V_i$ and consequently, $V_i$ is starshaped. 
\item In order to prove that $td_i+(1-t)x\notin V_j$ for every $j<i$, is enough to check that
\begin{equation}\label{starshapedness}
\Vert td_i+(1-t)x-d_i\Vert <\Vert td_i+(1-t)x-d_j\Vert. 
\end{equation}
For fixed $j<i$, we assume without loss of generality that $d_j=0$.
Define the function $F:\R\to \R$ by $F(\lambda)=\Vert x+\lambda d_i\Vert $ and observe that proving \eqref{starshapedness} is equivalent to proving that $F(\lambda)>\Vert x-d_i\Vert $ for $\lambda\geq 0$. Note that $x\in V_i\setminus V_j$, hence $\Vert x-d_i\Vert<\Vert x\Vert$. Therefore $F(-1)=\Vert x-d_i\Vert <\Vert x\Vert =F(0)$ and since $F$ is convex we automatically have that $F(\lambda)>F(-1)=\Vert x-d_i\Vert$ for every $\lambda \geq 0$, hence \eqref{starshapedness} is proved. 
\end{itemize}
\end{proof}

\begin{lem}\label{tiling of W_k}
Let $W$ be a separable Banach space and let $Q$ be a 
norm $1$ projection of $W$ onto a one codimensional subspace $\widetilde{W}$ of $W$.
Then there exists a tiling $\{H_j\}_{j\ge 0}$ of $W$, points 
$h_j\in H_j$ and $0<r<1<R_0$  such that :
\begin{enumerate}
\item $B_{W}(0,1)\subset H_0\subset B_{W}(0,R_0)$

\item $\Vert Qh_j\Vert\le 1-r$

\item $B_{W}(h_j,r)\subset H_j$

\item for all $h\in H_j$, $\Vert h-h_j\Vert\le 4+2\Vert Qh\Vert$.

\end{enumerate}

\end{lem}
\begin{proof}
We shall use the following elementary facts. 

\begin{fact}\label{tiling of R^2}
{\em Let $D=[-1,1]^2$, and define 
$$
U_0=\{(x,y):\,\vert x\vert+\vert y\vert\le 2\}\qquad
U_1=\{(x,y):\,0\le x\le 2,\, y\ge 0,\text{ and }x+y\ge 0\}
$$
$$
U_2=\{(x,y):\,(x,-y)\in U_1\}\quad
U_3=\{(x,y):\,(-x,y)\in U_1\}\quad
U_4=\{(x,y):\,(-x,-y)\in U_1\}.
$$
$\{U_0,U_1,U_2,U_3,U_4\}$ is a tiling of $\{(x,y)\in\R^2:\,\vert x\vert\le 2\}$ and there exists $a>0$, $b>0$,
and $0<r,\delta<1$ such that : 
\begin{enumerate}[(a)]
\item $D\subset U_0\subset 2D$

\item $(a,b)+rD\subset U_1\cap \{x,y):\,\vert y\vert \le 1\}$

\item If $\vert t\vert\le \delta b$, then $(a,t)+rD\subset U_0$.
\end{enumerate}
}
\end{fact}

The following picture should convince the reader that Fact \ref{tiling of R^2} holds, corresponding to the values $(a,b)=(\frac 3 2, \frac 5 6)$, $r=\frac 1 6$ and $\delta=\frac 1 5$. 

\begin{center}
\begin{tikzpicture}[scale=0.85]
\draw (8.5,3)  -- (11.5,0) -- (8.5,-3) -- (5.5,0) --  cycle;
\draw[dashed] (3,0) -- (14,0);
\draw[dashed] (8.5,-3.5) -- (8.5,3.5);
\draw[dashed] (3,1.5) -- (14,1.5);
\draw (5.5,3.5) -- (5.5,-3.5);
\draw (11.5,3.5) -- (11.5,-3.5);
\draw (10.5,1) -- (10.5,1.5) -- (11,1.5) -- (11,1) -- cycle;
\draw (10.5,0.5) -- (10.5,0) -- (11,0) -- (11,0.5) -- cycle;
\draw (8.5,3) -- (8.5,3.5);
\draw (8.5,-3) -- (8.5,-3.5);
\draw (10.75,1.25) -- (10.75,0.25);
\draw[dashed] (10,1.5) -- (10,-1.5) -- (7,-1.5) -- (7,1.5) -- cycle;
\node at (3.3,1.75) {1};
\node at (5.15,-0.25) {-2};
\node at (6.7,-0.25) {-1};
\node at (3.3,0.25) {0};
\node at (9,0.5) {\huge $U_0$};
\node at (10.5,2.3) {\huge $U_1$};
\node at (10.5,-2.3) {\huge $U_2$};
\node at (6.5,2.3) {\huge $U_3$};
\node at (6.5,-2.3) {\huge $U_4$};
\node at (10.75,1.35) {\scalebox{0.55}{(a,b)}};
\node at (8.05,3) {2};
\node at (11.8,-0.25) {2};
\node at (10.3,-0.25) {1};
\end{tikzpicture}
Figure 1
\end{center}

\begin{fact}\label{maximal family}
{\em
Given $0<\delta<1$, there exists $(v_j,v_j^*)\in \widetilde{W}\times \widetilde{W}^{*}$ such that
\begin{enumerate}[(a)]
\item $\Vert v_j\Vert=\Vert v_j^*\Vert=v_j^*(v_j)=1$ for all $j$.
\item $\vert v_j^*(v_k)\vert\le\delta$ if $j<k$.
\item $\sup\limits_j\vert v_j^*(v)\vert\geq\delta\Vert v\Vert$ for all $v\in \widetilde{W}$.
\end{enumerate} }
\end{fact}
\begin{proof}

Fix $0<\delta<1$. Take $\{V_n\}_{n\geq 1}$  an increasing sequence of subspaces with $\dim(V_n)=n$ and so that $\bigcup^{\infty}_{n=1}V_n$ is dense in $\widetilde{W}$. We will prove by induction that for each $n\in\N$ there is a maximal set $F_n=\{v_{1},\dots,v_{k_n}\}\subset V_n$ and $\{v^{*}_{1},\dots,v^{*}_{k_n}\}\subset V^*$ so that $(a)$ and $(b)$ are satisfied.

For $n=1$ just take a vector $v_1$ of norm one and $v_1\in \widetilde{W}^*$, with $v_1^*(v_1)=1$ and define $F_1=\{v_1\}$. Assuming the result is true for $n$, let us construct $F_{n+1}$. Starting from $F_n=\{v_1,\dots,v_{k_n}\}$ define a maximal set $F_{n+1}=F_n\cup\{v_{k_n+1},\dots,v_{k_{n+1}}\}$ in $V_{n+1}$, together with its supporting functionals so that properties $(a)$ and $(b)$ are satisfied. Note that $k_{n+1}$ must be finite because for every $j<i$ we have 

$$\Vert v_{j}-v_{i}\Vert \geq v^{*}_{j}(v_{j}-v_{i})\geq 1-\delta.$$

This means that $F_{n+1}$ is a set points in the unit sphere of a finite dimensional space that are $(1-\delta)$ separated. Hence by compactness this set of points must be finite. 

Consider $\bigcup^{\infty}_{n=1} F_n=\{v_j\}_{j\geq 1}$, $F_n\subset V_n$, and $\{v^{*}_{j}\}_{j\geq 1}$ their correspondent supporting functionals. Let us check that for every $v\in \widetilde{W}$

$$\sup\limits_j\vert v_j^*(v)\vert\geq\delta\Vert v\Vert.$$

By contradiction suppose there is some $v\in \widetilde{W}$ with $\Vert v\Vert =1$ and some $0<\delta'<\delta$ for which  $|v^{*}_{j}(v)|\leq\delta'$ for all $j\in\N$. By density of $\bigcup^{\infty}_{n=1}V_n$ in $\widetilde{W}$, there is some $n\in\N$ and some vector $w\in V_n$ with $\Vert w\Vert=1$ so that $\Vert w-v\Vert\leq (\delta-\delta')/ 2$. Hence for all $j\in\N$ we would have
$$|v^{*}_{j}(w)|\leq |v^{*}_{j}(v)|+|v^{*}_{j}(w-v)|<\delta,$$ 
which contradicts the maximality of the family $F_n$.
\end{proof}

We start now with the proof of Lemma \ref{tiling of W_k}.

\mn\rm Let $e\in \Ker (Q)$ and $e^*\in W^*$ such that 
$\Vert e\Vert=\Vert e^*\Vert=e^*(e)=1$. 
If $u\in \Ker (Q)$, then $u=e^*(u)e$. For each $w\in W$, since $w-Qw\in \Ker(Q)$, we have $w=Qw+e^*(w-Qw)e=Qw-e^*(Qw)e+e^*(w)e$.
Hence:
\begin{equation}\label{estimating the norm by the norm of the projection}
\Vert w\Vert\le\vert e^*(w)\vert+2\Vert Qw\Vert
\end{equation}

\noindent
Let $w_j^*=v_j^*\circ Q$ and $w_j=v_j-e^*(v_j)e$ : we have $Qw_j=v_j$ and $e^*(w_j)=0$.

\noindent
Let $\pi_j:W\to\R^2$ defined by $\pi_j(w)=(e^*(w),w_j^*(w))$.
We define
\begin{align*}
&H_0=\bigcap_j\pi_j^{-1}(U_0),\\
&H_j^p=\pi_j^{-1}(U_p)\cap\bigcap_{i<j}
\pi_i^{-1}(U_0)\;\;\text{for}\;1\le p\le 4\; \text{and }\;j\ge 1,\\
&H_n=\{w\in W:\,\vert e^*(w)-4n\vert\le 2\}\;\;\text{for}\;\;n\;\;\text{integer,}\;\; n\ne 0.
\end{align*}

The sets $H_0$ and $H_j^p$ for $1\le p\le 4$ and $j\ge 1$ 
form a tiling of $T=\{w\in W:\,\vert e^*(w)\vert\le 2\}$.
Indeed they are closed and convex as intersections of preimages 
of closed convex sets by continuous linear mappings.
They cover $T$ since for each $x\in T$, either $\pi_jx\in U_0$ for all $j$, and so $x\in H_0$,
or there exists $j$ such that  $\pi_jx\notin U_0$, and, if we take the smallest such $j$,
there exists $1\le p\le 4$ such that $\pi_jx\in U_p$, and we get $x\in H_j^p$.
The intersection of any two such sets has empty interior.
For instance, if $j<k$, then\quad 
$H_j^p\cap H_k^q\subset \pi_j^{-1}(U_p\cap U_0)$
has empty interior.
For $n$ integer, $n\ne 0$, the sets $H_n$ are closed and convex as preimages 
of closed intervals by a continuous linear functional. 
The sets $H_n$ for $n$ integer, and $H_j^p$ for $1\le p\le 4$ and $j\ge 1$ 
form a tiling of $W$. Indeed, they obviously form a covering of $W$,
and the intersection of $H_n$ with any other tile is included in $\{w\in W:\,\vert e^*(w)-4n\vert= 2\}$,
hence has empty interior.

\begin{enumerate}
\item We have $B(0,1)\subset H_0$ because for each $j$, $\pi_j(B(0,1))\subset [-1,1]^2\subset U_0$.

\noindent
We have $H_0\subset B(0,R_0)$ with $R_0=2+4/\delta$. If $w\in H_0$, then $\vert w_j^*(w)\vert\le 2$ for all $j$ (because $\pi_j(H_0)\subset U_0\subset 2D$), and hence
$$
\delta\Vert Qw\Vert \leq \sup_j |v^{*}_{j}(Qw)|=\sup_j |w^{*}_{j}(w)|\leq 2.
$$
Therefore
$$
\Vert w\Vert\le\vert e^*(w)\vert+2\Vert Qw\Vert\le 2+\frac{4}{\delta}
= R_0.
$$

\item If we define $h_j^1=ae+bw_j$, we have $\Vert Qh_j^1\Vert=b\Vert v_j\Vert=b\le 1-r$.

\item We also have $B_{W}(h_j^1,r)\subset H_j^1$. Indeed, $\pi_j(B(h_j^1,r))\subset (a,b)+rD\subset U_1$,
and, if $i<j$, $\pi_ih_j^1=(a,t)$ with $\vert t\vert= b\vert w_i^*(w_j)\vert=b\vert v_i^*(v_j)\vert\le\delta b$, thus
$\pi_i(B(h_j^1,r))\subset (a,t)+rD\subset U_0$.

\item Finally, if $h\in H_j^1$,
$$
\Vert h-h_j^1\Vert \le\vert e^*(h-h_j^1)\vert+2\Vert Q h_j^1\Vert + 2\Vert Qh\Vert
\le 2+2+2\Vert Qh\Vert.
$$
\end{enumerate}

The proofs of $(2),(3)$ and $(4)$ with $h_j^1$ replaced by $h_j^2=ae-bw_j$, 
$h_j^3=-ae+bw_j$ or $h_j^4=-ae-bw_j$ are the same. If we define, for $n\ne 0$, $h_n=4ne$,
we clearly have $Qh_n=0$, $B_{W}(h_n,r)\subset H_n$ and 
$\Vert h-h_n\Vert \le\vert e^*(h-h_n)\vert + 2\Vert Qh\Vert
\le 2+2\Vert Qh\Vert$, thus $(2),(3)$ and $(4)$ are also satisfied in this case. 
\end{proof}

\begin{proof}[Proof of Theorem \ref{tiling Schauder basis}]
\hspace{6cm}

We fix $(e_k)_{k\geq 0}$ a Schauder basis of $X$. For each $k$, we define 
$$V_k=\espan\{e_0,\cdots,e_{k}\}\; ; \; W_k=\overline{\espan}\{e_m: \,m> k\},$$ 
$P_k$ the projection of $X$ onto $V_k$ defined by 
$P_k(\sum\limits_{n=0}^\infty x_ne_n)=\sum\limits_{n=0}^{k} x_ne_n$,
and $Q_k=I-P_k$. Without loss of generality, we can assume that for each $k$, $\Vert Q_k\Vert=1$.
Indeed, it is enough to replace if necessary the norm of $X$ by the equivalent norm defined by
$\vert x\vert =\sup\{\Vert Q_k(x)\Vert:\,k\in\N\}$. We now fix $k\geq 0$.

\sn
By Lemma \ref{tiling of V_k}, let $\{D_i^k:\,i\geq 0\}$ be a tiling of the finite dimensional space $V_k$ and points 
$d_i^k\in D_i^k$ such that for each $i$, 
$B(d_i^k,r)\subset D_i^k\subset B(d_i^k,2r)$.

\sn
By Lemma \ref{tiling of W_k}, applied for $W=W_k$ and $\widetilde{W}=W_{k+1}$, let $\{H_j^k:\,j\geq 0\}$  be a tiling of the infinite dimensional space $W_k$ and points 
$h_j^k\in H_j^k$ such that:
\begin{enumerate}
\item $B_{W_k}(0,1)\subset H_0^k\subset B_{W_k}(0,R_0)$ 

\item $\Vert Q_{k+1}h_j^k\Vert\le 1-r$

\item $B_{W_k}(h_j^k,r)\subset H_j^k$

\item for all $h\in H_j^k$, $\Vert h-h_j^k\Vert\le 4+2\Vert Q_{k+1}h\Vert$.
\end{enumerate}

We now define the tiling of $X$, where $(i,j,k)\in\N^3$ satisfy $k\ge 1\Rightarrow j\ge 1$:
$$
C_{i,j}^k=P_k^{-1}(D_i^k)\cap Q_k^{-1}(H_j^k)\cap\bigcap_{m>k} Q_m^{-1}(H_0^m).
$$
The sets $C_{i,j}^k$ are closed because they are intersections of preimages of closed sets by continuous linear mappings. The sets $H_j^k$ and $H_0^m$ are convex, but the sets $D_i^k$ are only starshaped. Since $C^{k}_{i,j}=D^{k}_{i}+(H^{k}_{j}\cap\bigcap_{m>k}Q^{-1}_{k,m}(H^{m}_{0}))$, where $Q_{k,m}=Q_m|_{W_k}$, the sets $C^{k}_{i,j}$ will be starshaped because they are a sum of a starshaped set 
 with a convex set.

We now define $c_{i,j}^k:=d_i^k+h_j^k$.

\begin{itemize}
\item We have $C_{i,j}^k\subset B(c_{i,j}^k,R)$, with $R=2r+4+2R_0$. Indeed, if $x\in C_{i,j}^k$,
condition $(4)$ and $Q_{k+1}x\in H_0^{k+1}\subset B_{W_{k+1}}(0,R_0)$ imply:
$$
\Vert x-c_{i,j}^k\Vert\le\Vert P_kx-d_i^k\Vert+\Vert Q_kx-h_j^k\Vert\le
2r+4+2\Vert Q_{k+1}x\Vert =2r+4+2R_0
$$

\item Let us now prove that $B(c_{i,j}^k,r)\subset C_{i,j}^k$.

\noindent
$P_k(c_{i,j}^k)=d_i^k\quad\Rightarrow\quad 
P_k(B(c_{i,j}^k,r))\subset B_{V_k}(d_i^k,r)\subset D_i^k$,

\noindent
$Q_k(c_{i,j}^k)=h_j^k\quad\Rightarrow\quad 
Q_k(B(c_{i,j}^k,r))\subset B_{W_k}(h_j^k,r)\subset H_j^k$.
Then, if $m>k$, $\Vert Q_m(c_{i,j}^k)\Vert=\Vert Q_m(h_j^k)\Vert=
\Vert Q_mQ_{k+1}(h_j^k))\Vert\le\Vert Q_{k+1}(h_j^k)\Vert\le 1-r
\,\,\,\Rightarrow\,\,\, Q_m(B(c_{i,j}^k,r))\subset B(0,1)\subset H_0^m$.

\item The sets $C_{i,j}^k$ cover $H$. Let us fix $x\in H$.
Since $\Vert Q_k(x)\Vert\to 0$, there exists $k$ such that for every $m>k$,
$Q_mx\in H_0^m$. We choose $k$ minimal with respect to this property.
Observe that if $k\ge 1$, then $Q_kx\notin H_0^k$.
Since $\bigcup\limits_i D_i^k=V_k$, there exists $i$ such that $P_kx\in D_i^k$.

\noindent 
Since $\bigcup\limits_j H_j^k=W_k$, there exists $j$ such that $Q_kx\in H_j^k$,
and if $k\ge 1$, then $j\ge 1$.

\item The sets $C_{i,j}^k\cap C_{p,q}^m$ have empty interior if $(i,j,k)\ne(p,q,m)$. Indeed,

\noindent
If $k=m$ and $j=q$ and $i\ne p$, \quad 
$C_{i,j}^k\cap C_{p,j}^k\subset P_k^{-1}(D_i^k\cap D_p^k)$
has empty interior.

\noindent
If $k=m$ and $j\ne q$, \quad 
$C_{i,j}^k\cap C_{p,q}^k\subset Q_k^{-1}(H_j^k\cap H_q^k)$
has empty interior.

\noindent
If $k<m$, necessarily $q>0$ and\quad 
$C_{i,j}^k\cap C_{p,q}^m\subset Q_m^{-1}(H_0^m\cap H_q^m)$
has empty interior.
\end{itemize}
\end{proof}

\begin{rem} {\em
The Figure 1 in Fact \ref{tiling of R^2} gave us the values $r=\frac 1 6$ and $\delta=\frac 1 5$, so $R_0=2+{4\over\delta}=22$ and $R=2r+4+2R_0={145\over 3}$. Then we have the relation ${R\over r}=290$.

If our aim was to get a better relation $R/r$, we can improve the Figure 1 by letting $(a,b)=(5.25,0.75)$, $r=0.25$, $\delta=0.25$ and $U_0=\{(x,y):\,|x|+8|y|\leq 9,\,|x|\leq 5.5\}$, $U_1=\{(x,y):\,x\geq 0, y\geq 0, x+8y\geq 9\}$ (and $U_2,U_3,U_4$ similarly). The following picture represents this.


\begin{center}
\begin{tikzpicture}[scale=0.96]
\draw (8.5,1.6363)  -- (16.5,0.6363) -- (16.5, -0.6363) -- (8.5,-1.6363) -- (0.5,-0.6363) -- (0.5,0.6363) -- cycle;
\draw[dashed] (0,0) -- (17,0);
\draw[dashed] (8.5,-2.5) -- (8.5,2.5);
\draw[dashed] (0,1.4545) -- (17,1.4545);
\draw (0.5,2.5) -- (0.5,-2.5);
\draw (16.5,2.5) -- (16.5,-2.5);
\draw (16.5,0.72725) -- (16.5,1.4545) -- (15.77275,1.4545) -- (15.77275,0.72725) -- cycle;
\draw (16.5,0.6363) -- (16.5,-0.09095) -- (15.77275,-0.09095) -- (15.77275,0.6363) -- cycle;
\draw (8.5,1.6363) -- (8.5,2.5);
\draw (8.5,-1.6363) -- (8.5,-2.5);
\draw (16.136375,1.090875) -- (16.136375,0.272675);
\draw[dashed] (9.9545,1.4545) -- (9.9545,-1.4545) -- (7.0455,-1.4545) -- (7.0455,1.4545) -- cycle;
\node at (0,1.7) {1};
\node at (0,0.25) {0};
\node at (9,0.5) {\huge $U_0$};
\node at (13,2.3) {\huge $U_1$};
\node at (13,-2.3) {\huge $U_2$};
\node at (4,2.3) {\huge $U_3$};
\node at (4,-2.3) {\huge $U_4$};
\node at (16.136375,1.24) {\scriptsize (a,b)};
\node at (7.95,1.9) {1.125};
\node at (16.85,-0.25) {5.5};
\node at (10.17,-0.3) {1};
\end{tikzpicture}
Figure 2
\end{center}
\vspace{0.3cm}

In this manner, some of the estimates in the previous proof are slightly modified. Using inequality \eqref{estimating the norm by the norm of the projection} we get $R_0=5.5+2\left( {1.125\over 0.25}\right) =14.5$, and condition $(4)$ in Lemma \ref{tiling of W_k} becomes $\Vert h-h_j\Vert\leq 5.5 +2+\Vert Qh\Vert$ for all $h\in H_j$, so $R=2r+7.5+2R_0=37$ and $R/r=148$. Moreover, in the particular case that we work with a $1$-unconditional basis (note that any Banach space with unconditional basis can be renormed so that it is $1$-unconditional), inequality \eqref{estimating the norm by the norm of the projection} becomes 
$$\Vert w\Vert \leq|e^*(w)|+\Vert Qw\Vert$$ and therefore $R_0=5.5+{1.125\over 0.25}=10$, $R=2r+6.5+R_0=17$ and $R/r= 68$.

Let us mention that David Preiss for the case of the separable Hilbert space is able to get $R/r=15$.}
\end{rem}

\section{Tiling of the unit sphere}\label{Tiling of the unit sphere}

\begin{defn}
{\em
Let $X$ be a Banach space.  A tiling of a subset $F\subset X$, is a family $\{T_{\alpha}\}\subset F$ of closed sets with non-empty relative interior in $F$, that covers $F$, and such that, if $\alpha\ne \beta$, then $T_\alpha\cap T_\beta$ has empty  relative interior with respect to $F$.

$F\subset X$ is said to admit normal tilings if for every $\varepsilon>0$ there exists $\rho>0$ and a tiling $\{T_\alpha\}$ of $F$, so that for some suitable points $c_\alpha\in T_\alpha$ we have
$$B(c_\alpha,\rho)\cap F\subset T_\alpha\subset B(c_\alpha,\varepsilon)\cap F.$$
}
\end{defn}

\begin{defn}
{\em A tiling  $\{S_{\alpha}\}$ of the sphere $S_X=\{x\in X:\,\Vert x\Vert=1\}$ will be called convex if there exists a convex tiling of $X$, $\{C_\alpha\}$, so that $S_\alpha = C_\alpha\cap S_X$.

}
\end{defn}

\begin{defn}
{\em 
A Banach space $X$ is uniformly convex if the modulus of convexity
$$\delta(\varepsilon)=\inf\{1-\Big\Vert {x+y\over 2}\Big\Vert:\,\Vert x\Vert,\Vert y\Vert\leq 1,\,\Vert x-y\Vert \geq \varepsilon\}$$
satisfies $\delta(\varepsilon)>0$ for every $\varepsilon>0$.
In particular if $x,y\in B_X$ are so that $\Vert {x+y\over 2}\Vert \geq 1-\delta(\varepsilon)$ then  $\Vert x-y\Vert\leq\varepsilon$.}
\end{defn}

\begin{thm}\label{tiling the sphere for uniformly convex spaces}
Let $X$ be a separable uniformly convex Banach space. Then the unit sphere $S_X=\{x\in X:\,\Vert x\Vert=1\}$ admits convex normal tilings.
\end{thm}
\begin{proof}
Take $\varepsilon\in (0,1)$, let $\delta=\delta(\varepsilon /2)>0$ be the modulus of convexity and fix $(1-2\delta)<r'<r<1$.

According to Fact \ref{maximal family} there is a maximal family $(x_j,f_j)\subset X\times X^*$ so that $\Vert x_j\Vert=\Vert f_j\Vert =f_j(x_j)=1$ and $|f_j(x_i)|\leq r$ if $j< i$.  In particular $\sup_{i}|f_{i}(x)|\geq r\Vert x\Vert$. \\

Name $U_0=[-r',r']$,  $U_1=(-\infty,-r']$ and $U_2=[r',\infty)$, and define
\begin{align*}
&H_0=\bigcap_{i\geq 1}f^{-1}_{i}(U_0)\\
&H^{p}_{j}=f^{-1}_{j}( U_p)\cap\bigcap_{i<j}f^{-1}_{i}(U_0)\;\;\text{for}\; p=1,2\;\text{ and}\;j\geq 1 .
\end{align*}
It is easy to see that these sets form a convex tiling of $X$ with unbounded tiles $H^{p}_{j}$ and $B(0,r')\subset H_0\subset B(0,{r'\over r}).$
To check the second inclusion note that if $x\in H_0$ then  $\Vert x\Vert\leq {1\over r} \sup_{i}|f_{i}(x)|\leq{r'\over r}$.\\

Define $\rho=r'{1-r\over 1+r}$ and observe that if we let $r,r'$ be very close to $1-2\delta$ then $\rho\sim {1-2\delta\over 1-\delta}\delta$ behaves asymptotically as $\delta=\delta({\varepsilon\over 2})$, whenever $\varepsilon$ tends to zero. We claim that the sets $H^{p}_{j}\cap S_X$ form a convex tiling of $S_X$ and that for all $(j,p)\in \N^*\times\{1,2\}$ and suitable $h^{p}_{j}\in S_X$ we have
$$B(h^{p}_{j},\rho)\cap S_X\subset H^{p}_{j}\cap S_X\subset B(h^{p}_{j},\varepsilon)\cap S_X. $$ 

Let us start by defining $R={2r'\over 1+r}\in(r',{r'\over r}) $. The centres $\{h^{p}_{j}\}$ of our tiling are defined as follows: take $v_j\in \bigcap^{j}_{i=1} \Ker(f_i)$ with $\Vert Rx_j+v_j\Vert=1$, and let
\begin{align*}
h^{1}_{j}=-Rx_j-v_j\;;\;h^{2}_{j}=Rx_j+v_j.
\end{align*}
We have $\Vert h^{p}_{j}\Vert=1$, $f_j(h^{p}_{j})=(-1)^{p}R$ and $|f_i(h^{p}_{j})|\leq r R<1$ for $i<j$. The following properties hold:
\begin{enumerate}
\item  $\{H^{p}_{j}\}$ is a covering of $S_X$ and they have pairwise disjoint relative interiors in $S_X$.

For the covering property just observe that $H_0\cap S_X=\emptyset$. Because, as  we already explained above, if $x\in H_0$ then $\Vert x\Vert\leq{r'\over r}<1$. Also the sets $H^{p}_{j}\cap H^{q}_{i}\cap S_X$ have empty relative interior in $S_X$ if $(p,j)\neq (q,i)$. Indeed,
\begin{itemize}

\item If $i=j$ and $p\neq q$, $H^{p}_{j}\cap H^{q}_{j} \subset f^{-1}_{j}((-1)^p r')\cap f^{-1}_{j}((-1)^q r')=\emptyset.$

\item If $i>j$, $H^{p}_{j}\cap H^{q}_{i}\cap S_X \subset f^{-1}_{j}((-1)^p r')\cap S_X$ has empty relative interior. Otherwise there is $x\neq y$, $\Vert x\Vert=\Vert y\Vert=1$, $f_j(x)=f_j(y)=(-1)^p r'$ for which ${(x+y)\over \Vert x+y\Vert }\in f^{-1}_{j}((-1)^p r')\cap  S_X$, and this means that $f_j(x+y)=(-1)^p\ r'\Vert x+y\Vert=2(-1)^p r'$, so $\Vert x+y\Vert=2$, contradicting  the strict convexity of the space $X$.

\end{itemize}
\medskip

\item $B(h^{p}_{j},\rho)\subset H^{p}_{j}$.

Since $\rho=R-r'$ we have that if $\Vert y-h^{p}_{j}\Vert\leq \rho$ then,
\begin{align*}
&f_j(y)=f_j(y-h^{1}_{j})+f_j(h^{1}_{j})\leq \rho-R=-r'\;\Rightarrow\;f_j(y)\in U_1\\
&f_j(y)=f_j(y-h^{2}_{j})+f_j(h^{2}_{j})\geq -\rho+R=r'\;\Rightarrow\;f_j(y)\in U_2.
\end{align*}
Also for $p=1,2$, $i< j$,
$$|f_i(y)|=|f_i(y-h^{p}_{j}) +f_i(h^{p}_{j})|\leq \rho+r R= r'\;\Rightarrow\;f_i(y)\in U_0.$$

\item $S_X \cap H^{p}_{j}\subset B(h^{p}_{j},\varepsilon)$

It is enough to prove that $S_X \cap H^{p}_{j}\subset B((-1)^p x_j,{\varepsilon\over 2})$.
Let us take $\Vert x\Vert =1$ with $x\in H^{p}_{j}$ and by contradiction assume $\Vert x-(-1)^p  x_j\Vert\geq{\varepsilon\over 2}$. In that case by uniform convexity we know that $\Vert x+(-1)^p  x_j\Vert\leq 2(1-\delta).$
For the case $p=1$,
$$f_j( x)=f_j\left( x_j\right) +f_j\left( x-x_j\right) \geq 1-2(1-\delta)=(2\delta-1)>-r'$$
contradicting that $x\in H^{1}_{j}$. And for the case $p=2$,
$$f_j( x)=f_j\left( -x_j\right) + f_j\left( x+x_j\right) \leq -1+2(1-\delta)=(1-2\delta)<r'$$
contradicting that $x\in H^{2}_{j}$.
\end{enumerate}

\end{proof}


Notice that the assumption of separability is superfluous in the preceding argument  because the existence of such a maximal family $(x_\alpha,f_\alpha)\in X\times X^*$ from Fact \ref{maximal family} is also true for $X$ nonseparable (see \cite[Lemma 2.1]{FPZ} for a proof of this fact), and the rest of the proof follows naturally.



\section{Tiling of convex bodies}\label{Tiling of convex bodies}

A convex {\em body} $C\subset X$ is a closed bounded convex set with nonempty interior. In this case note that $C$ is said to admit normal tilings if for every $\varepsilon>0$ there exists $\rho>0$ and a tiling $\{T_\alpha\}$ of $C$, so that for some suitable points $c_\alpha\in T_\alpha$ we have
$$B(c_\alpha,\rho)\subset T_\alpha\subset B(c_\alpha,\varepsilon).$$

For a convex body $C$, a {\em slice} of it will be any non-empty closed convex subset of the form $\{x\in C: f(x)\geq \lambda\}$ where $f\in S_{X^*}$ and $\lambda\in\R$.

\begin{thm}\label{main theorem}
Let $X$ be an uniformly convex Banach space. Then every convex body $C$ admits convex normal tilings.
\end{thm}

We will first need two auxiliary results.

\begin{lem}\label{building the tiles}
Let $\varepsilon,\eta\in (0,1)$ and $\delta=\delta(\varepsilon)>0$ the modulus of convexity. Then for every convex set $C$ where $B(0,\eta)\subset C\subset B(0,1)$ and for every $x\in C$ with $\Vert x\Vert\geq \sqrt{1-\delta}$, there exists a slice $T$ containing $x$, containing a ball of radius $r=r(\delta,\eta)>0 $,  $diam (T)<\varepsilon$ and $T\cap \buildrel\circ\over B(0,1-\delta)=\emptyset$.
\end{lem}
\begin{proof}[Proof of Lemma \ref{building the tiles}]
Let $\eta_\delta=\min\{\eta,1-\delta\}>0$ and define 
\begin{align*}
&x_0={x\over \Vert x\Vert}\sqrt{1-\delta} \;\;\;\;;
&y_0=\left( {\eta_\delta+(1-\delta)\over \eta_\delta+\sqrt{1-\delta}}\right) x_0\;\;;\\
&r=\eta_\delta{\sqrt{1-\delta}-(1-\delta)\over \eta_\delta+\sqrt{1-\delta}}>0.
\end{align*} 
Obviously $x_0,y_0\in C$ by convexity of $C$.
We have the following properties:
\begin{enumerate}
\item $\buildrel\circ\over B(0,1-\delta)\cap B(y_0,r)=\emptyset$ .

Just note that
$$\Vert y_0\Vert-r=\sqrt{1-\delta}\,{\eta_\delta+(1-\delta)\over \eta_\delta+\sqrt{1-\delta}}-\eta_\delta\,{\sqrt{1-\delta}-(1-\delta)\over \eta_\delta+\sqrt{1-\delta}}=(1-\delta). $$
\item $B(y_0,r)\subset C$.

Indeed, since $x_0\in C$ and $B(0,\eta_\delta)\subset C$, by convexity it is enough to check that $h(B(y_0,r))\subset B(0,\eta_\delta)$, where $h$ is a suitable homothety $h:X\to X$.
Explicitly $h(y)=x_0+K(y-x_0)$ for $y\in X$, where $K={\eta_\delta\over r}={\eta_\delta+\sqrt{1-\delta}\over \sqrt{1-\delta}-(1-\delta)}>1$. Now take $y\in B(y_0,r)$ and write 
\begin{align*}
\Vert h(y)\Vert&=\Vert x_0+K(y-x_0)\Vert\leq\Vert x_0+K(y_0-x_0)\Vert+K\Vert y-y_0\Vert\\
&=\Vert x_0\Vert\left[1+K\left({\eta_\delta+(1-\delta)\over \eta_\delta+\sqrt{1-\delta}}-1\right) \right]+K\Vert y-y_0\Vert=0+K\Vert y-y_0\Vert\leq Kr=\eta_\delta.
\end{align*}
\end{enumerate}

To conclude the proof, by Hahn Banach separation theorem there is some hyperplane separating $B(y_0,r)$ and $B(0,1-\delta)$. This hyperplane defines a slice $T$ of $C$ including $x$, containing a ball of radius $r$ (that is $B(y_0,r)$) and with empty intersection with $\buildrel\circ\over B(0,1-\delta)$. It also satisfies that $diam \,(T)\leq \varepsilon$ since for any two points $x,y\in T$ we have $\Vert {x+y\over 2}\Vert\geq 1-\delta  $ and hence $\Vert x-y\Vert\leq \varepsilon$.

\end{proof}

\begin{lem}\label{transfinite induction}
Let $\varepsilon,\eta\in (0,1)$ and $\delta=\delta(\varepsilon)>0$ the modulus of convexity. Then for all closed convex set $C$, so that $B(0,\eta)\subset C\subset B(0,1)$, there exists $r=r(\delta,\eta)>0$, and a transfinite decreasing sequence of closed convex sets $\{D_\alpha\}_{\alpha\leq\mu}$ defined as:
\begin{enumerate}
\item $D_0=C$.
\item For $\alpha$ a limit ordinal $D_\alpha=\bigcap_{\beta<\alpha}D_\beta$.
\item For every $\alpha$, $D_{\alpha+1}=\overline{D_\alpha\setminus T_\alpha}$ where $T_\alpha$ is a slice of $D_\alpha $ such that
\begin{itemize}

\item For some  $c_\alpha$ we have $B(c_\alpha,r)\subset T_{\alpha}\subset B(c_\alpha,\varepsilon).$
\item $T_\alpha\setminus B(0,\sqrt{1-\delta})\neq\emptyset$.
\item $T_\alpha\,\cap \buildrel\circ\over B(0,1-\delta)=\emptyset$.
\end{itemize} 
\end{enumerate}
so that $B(0,\min\{\eta,1-\delta\})\subset D_\mu\subset B(0,\sqrt{1-\delta})$.

\end{lem}

\begin{proof}[Proof of Lemma \ref{transfinite induction}]
Fix $\eta_\delta=\min\{\eta,1-\delta\}>0$ and let $r=r(\delta,\eta)>0$ be given by Lemma \ref{building the tiles}. We proceed by transfinite induction.

Suppose that $D_\beta$ has been defined for every $\beta\leq\alpha$.
If there is no $x\in D_\alpha$ with $\Vert x\Vert>\sqrt{1-\delta}$ define $D_{\alpha+1}=D_\alpha$ and stop.
Otherwise, take $x\in D_\alpha$ with $\Vert x\Vert>\sqrt{1-\delta}$ and apply Lemma \ref{building the tiles} to get a slice $T_{\alpha}$ of $D_\alpha$ including $x$, containing a ball of radius $r>0$, with diameter less than $\varepsilon$ and with disjoint intersection with $\buildrel\circ\over B(0,1-\delta)$. Define then $D_{\alpha+1}=\overline{D_\alpha\setminus T_{\alpha}}$, which satisfies clearly condition $(3)$.

The end of this transfinite sequence follows from property $(3)$ since at every step where $\alpha$ is a successor ordinal we are removing points $\Vert x\Vert>\sqrt{1-\delta}$, and hence we also get $B(0,\eta_\delta)\subset D_\mu\subset B(0,\sqrt{1-\delta})$.
\end{proof}

\begin{proof}[Proof of Theorem \ref{main theorem}]
 Without loss of generality we will assume that our convex body satisfies $B(0,\eta)\subset C\subset B(0,1)$.
For a given $\varepsilon\in (0,1)$ let $\delta=\delta(\varepsilon)>0$ be the modulus of convexity. For simplicity let as introduce the notation $\gamma=\sqrt{1-\delta}$ and take $n\in\N$ so that $\gamma^{n}\leq\varepsilon$. 

\begin{claim} \label{claim of the main proof for uniform convex spaces}
There is a decreasing sequence of closed convex sets $C=C_1\supset C_2\supset\cdots\supset C_n\supset C_{n+1}$ so that for every $k=1,\dots,n+1$,
\begin{enumerate}[(a)]
\item $B(0,\min\{\eta,\gamma^{k}\})\subset C_{k}\subset B(0,\gamma^{k-1}).$
\item For $k=1,\dots,n$ there exists a tiling of $\overline{C_k\setminus C_{k+1}}$ by slices $\{T^{k}_{\alpha}\}$ of $C_k$ so that $T^{k}_{\alpha}\cap T^{k}_{\beta}$ has empty interior for $\alpha\neq\beta$, and for some $c^{k}_{\alpha}$,
$$B(c^{k}_{\alpha},r_k\gamma^{k-1})\subset T^{k}_{\alpha}\subset B(c^{k}_{\alpha},\varepsilon).$$
\end{enumerate}
\end{claim}
Assuming the claim is true we conclude the proof by defining $\rho=\min\{\eta,\gamma^{n+1},\min_{1\leq k\leq n}\{r_{k}\gamma^{k-1}\}\}$ and noticing that $ \{T^{k}_{\alpha}\} \cup C_{n+1}$ gives the required tiling of $C$.
\begin{itemize}
\item They are closed convex sets whose intersections lie in affine hyperplanes (by construction), hence they have empty interior.
\item $B(c^{k}_{\alpha},\rho)\subset T^{k}_{\alpha}\subset B(c^{k}_{\alpha},\varepsilon)$ because $\rho\leq r_k\gamma^{k-1}$ for every $k=1,\dots,n$.
\item $B(0,\rho)\subset B(0,\min\{\eta,\gamma^{n+1}\})\subset C_{n+1}\subset B(0,\gamma^n)\subset B(0,\varepsilon)$.
\end{itemize}

\end{proof}
\begin{proof}[Proof of Claim \ref{claim of the main proof for uniform convex spaces}]
We use induction and successive applications of Lemma \ref{transfinite induction}.
 
\medskip

$\blacktriangleright$ {\bf Step 1}: Given $\varepsilon,\delta$ and $C_1=C$ apply Lemma \ref{transfinite induction} to get a number $r_1>0$, a decreasing sequence of closed convex sets $\{D^{1}_{\alpha}\}\subset C_1$ and a sequence of vectors $\{c^{1}_{\alpha}\}$ satisfying $(1)$--$(3)$ of the lemma. Define $T^{1}_{\alpha}=\overline{D^{1}_{\alpha}\setminus D^{1}_{\alpha +1}}$ and  $\bigcap_{\alpha} D^{1}_{\alpha} =C_2$. We have
$$B(c^{1}_{\alpha},r_1)\subset T^{1}_{\alpha}\subset B(c^{1}_{\alpha},\varepsilon)\;\;;\;\;C_1\setminus C_2\subset \bigcup_{\alpha}T^{1}_{\alpha}\;\;;\;\; B(0,\min\{\eta,\gamma^2\})\subset C_2\subset B(0,\gamma).$$
And also the sets $T^{1}_{\alpha}\cap T^{1}_{\beta}$ for $\alpha\neq\beta$ have empty interior since they lie inside hyperplanes.

$\blacktriangleright$ {\bf Induction step}: Assume we have already defined the closed convex set $C_{k}$ so that
$$B(0,\min\{\eta,\gamma^{k}\})\subset C_{k}\subset B(0,\gamma^{k-1}).$$
Take $C'_{k}={1\over \gamma^{k-1}}C_k$. Note that $B(0,\min\{{\eta\over\gamma^{k-1}},\gamma\})\subset C'_k\subset B(0,1)$. Then apply Lemma \ref{transfinite induction} for $C'_k$ getting then a number $r_k>0$, a decreasing sequence of closed convex sets $\{(D^{k}_{\alpha})^{'}\}$ and a sequence of vectors $\{(c^{k}_{\alpha})^{'}\}$ satisfying $(1)$--$(3)$ of the lemma. Let us call $(T^{k}_{\alpha})^{'}=\overline{(D^{k}_{\alpha})^{'}\setminus (D^{k}_{\alpha+1})^{'}}$ and  $\bigcap_{\alpha} (D^{k}_{\alpha})^{'} =C'_{k+1}$. We have 
$$B((c^{k}_{\alpha})^{'},r_k)\subset (T^{k}_{\alpha})^{'}\subset B((c^{k}_{\alpha})^{'},\varepsilon)\;\;;\;\;C'_k\setminus C'_{k+1}\subset \bigcup_{\alpha}(T^{k}_{\alpha})^{'}\;\;;\;\; B(0,\min\{{\eta\over \gamma^{k-1}},\gamma,\gamma^2\})\subset C'_{k+1}\subset B(0,\gamma).$$

Multiplying by $\gamma^{k-1}$, if we define $c^{k}_{\alpha}=\gamma^{k-1}(c^{k}_{\alpha})^{'}$, $C_{k+1}=\gamma^{k-1} C'_{k+1}$ and $T^{k}_{\alpha}=\gamma^{k-1}(T^{k}_{\alpha})^{'}$ we get

$$B(c^{k}_{\alpha},r_k\gamma^{k-1})\subset T^{k}_{\alpha}\subset B(c^{k}_{\alpha},\varepsilon)\;\;;\;\;C_k\setminus C_{k+1}\subset \bigcup_{\alpha}T^{k}_{\alpha}\;\;;\;\; B(0,\min\{\eta,\gamma^{k+1}\})\subset C_{k+1}\subset B(0,\gamma^k).$$
And also the sets $T^{k}_{\alpha}\cap T^{k}_{\beta}$ for $\alpha\neq\beta$ have empty interior since they lie inside hyperplanes.
\end{proof}

\section{Normal tiling of the unit ball of $L_1$}\label{Normal tiling of the unit ball of $L_1$}

By an uniform homeomorphism  we mean a bijective uniform continuous mapping whose inverse is also uniform continuous.

\begin{lem}\label{uniform homeomorphisms preserve normality}
Let $X,Y$ be two Banach spaces and $F\subset X$, $G\subset Y$ uniformly homeomorphic. Then $F$ admits normal tilings if and only if $G$ admits normal tilings.
\end{lem}
\begin{proof}
Let $h:F\to G$ be an uniform homeomorphism. We can define the modulus of continuity $\omega:(0,\infty)\to(0,\infty)$ for $h$ and $h^{-1}$ all together as
$$\omega(\delta)=\sup \left\lbrace  \sup\{\Vert h(x)-h(x')\Vert:\,\Vert x-x'\Vert\leq\delta\},\sup \{\Vert h^{-1}(y)-h^{-1}(y')\Vert:\,\Vert y-y'\Vert \leq\delta\}\right\rbrace  ,$$
where $\omega(\delta)\to 0$ as $\delta\to 0$.
Therefore we can write
\begin{equation}\label{uniform homeomorphisms}
\left\lbrace \begin{array}{l}
h(B(x,\delta))\subset B(h(x),\omega(\delta))\;\;\text{for all}\;x\in F; \vspace{0.3cm}\\
h^{-1}(B(y,\delta))\subset B(h^{-1}(y),\omega(\delta))\;\;\text{for all}\;y\in G  \Leftrightarrow\; B(h(x),\delta)\subset h(B(x,\omega(\delta)) \;\;\text{for all}\;x\in F.
\end{array}
\right.
\end{equation}

Assume now that $F$ admits normal tilings. Then there exists a tiling $\{C_\alpha\}$ of $F$ so that for some $R>r>0$ and some centres $\{c_\alpha\}$ we have
$$B(c_\alpha,r)\cap F\subset C_\alpha\subset B(c_\alpha,R)\cap F.$$
Since a homeomorphism sends sets with empty interior to sets with empty interior, by letting $\rho>0$ so that $\omega(\rho)\leq r$ and using \eqref{uniform homeomorphisms} we conclude that  $\{h(C_\alpha)\}$ is a tiling of $G$ such that
$$B(h(c_\alpha),\rho)\cap G\subset h(C_\alpha)\subset B(h(c_\alpha),\omega(R))\cap G .$$
\end{proof}

\begin{cor}\label{tiling the ball of l_1}
Let $\mu$ be a measure. Then the unit ball and the unit sphere of $L_1(\mu)=L_1$ admit normal tilings.
\end{cor}
\begin{proof}
Define for $1\leq q<\infty$,
\begin{align*}
M:&\; L_2(\mu)\longrightarrow L_q(\mu)\\
&f(x)\longmapsto \text{sign}(f(x))|f(x)|^{2/q}.
\end{align*}
This mapping is called the Mazur map and was introduced in \cite{Mazur}. It was used to prove that the unit spheres of $L_2(\mu)$ and $L_q(\mu)$ were uniformly homeomorphic. The same mapping defines an uniform homeomorphism between the corresponding unit balls.
Furthermore for the case of $q=1$, $M$ is Lipschitz in the unit ball of $L_2(\mu)$, and  $M^{-1}$ is Hölder continuous in the whole $L_1(\mu)$. Namely (for the case of real valued functions),

\begin{align*}\label{Mazur map inclusions}
\left\lbrace \begin{array}{l}
\Vert M(f)-M(g)\Vert _1\leq 2\Vert f-g\Vert_2 \;\;\text{for all}\; f,g\in B_{L_2}(0,1) \vspace{0.3cm}\\
\Vert M^{-1}(f)-M^{-1}(g)\Vert_2\leq 2\Vert f-g\Vert^{1/2}_{1} \;\;\text{for all}\; f,g\in L_1(0,1).
\end{array}
\right.
\end{align*}
Consequently, since $L_2(\mu)$ admits normal tilings, and using Lemma \ref{uniform homeomorphisms preserve normality}, we conclude that the unit ball and unit sphere of $L_1(\mu)$ admit also normal tilings.
Note that in general the tiles making up the normal tiling of $B_{L_1}(0,1)$ will not be convex, nor starshaped.
\end{proof}

An important fact is that, contrary to what happens in the unit sphere and unit ball, the spaces $L_p(\mu)$ and $L_q(\mu)$ are not uniformly homeomorphic if $p\neq q$ (see \cite{Lin, Enflo}). To know more properties about the Mazur map and uniform homeomorphisms between Banach spaces we suggest the reader to see \cite[Chapter 9]{BenLin}.

\end{document}